\documentclass{article}

\usepackage{fullpage}
\usepackage{verbatim}
\usepackage{graphics}
\usepackage{graphicx}
\usepackage{amsmath, amssymb, amsfonts, amsthm}
\usepackage{url}
\usepackage{comment}
\usepackage{placeins}

\newtheorem{thm}{Theorem}

\newtheorem{conj}[thm]{Conjecture}

\newtheorem{prob}{Problem}

\begin{document}

\author{Ben D. Lund\footnote{Department of Computer Science, Rutgers University, lundbd@gmail.com}\\George B. Purdy\footnote{Department of Electrical Engineering and Computing Systems, University of Cincinnati, george.purdy@uc.edu}\\Justin W. Smith\footnote{Department of Computer Science, Northern Kentucky University, justin.smith@nku.edu}}

\title{A Pseudoline Counterexample to the Strong Dirac Conjecture}
\maketitle

\begin{abstract}
We demonstrate an infinite family of pseudoline arrangements,
in which an arrangement of $n$ pseudolines has no member incident
to more than $4n/9$ points of intersection.
This shows the ``Strong Dirac'' conjecture to be false for pseudolines.

We also raise a number of open problems relating to possible differences between 
the structure of incidences between points and lines 
versus
the structure of incidences between points and pseudolines.
\end{abstract}

\section{Introduction}\label{sec:introduction}

A central problem of discrete geometry is to elucidate the structure of incidences between points and lines.
Until the recent explosion of applications of polynomial methods to problems in incidence geometry (\cite{Mat11},\cite{tao2013algebraic}, \cite{dvir2012incidence}), the tools most successfully applied to questions about incidences 
between points and lines 
could be immediately applied to prove equivalent results for incidences 
between points and pseudolines.
By an arrangement of pseudolines, we mean a set of simple closed curves in the real projective plane, any pair of which meet at a single crossing point \cite[p. 79]{Fel04}.

Given an arrangement $L$ of lines in the real projective plane, $\mathbb{P}^2$, 
let $r(L)$ be the maximum number of vertices on any line of $L$ (a vertex of $L$ is a point incident to at least $2$ lines of $L$).
In 1951, G. Dirac (working in the dual context of point sets) conjectured a lower bound on $r(L)$ \cite{Dir51}.
\begin{conj}[Strong Dirac]\label{conj:strongDirac}
Let $L$ be an arrangement of $n$ lines in $\mathbb{P}^2$ that do not all pass through a single point.
There exists a constant $c$ such that
\[r(L) \geq n/2 - c.\]
\end{conj}

In this paper, we show that the old and widely believed ``Strong Dirac" conjecture is false when generalized from arrangements of lines to arrangements of pseudolines.
We give an explicit example of an infinite family of pseudoline arrangements $\mathcal{L}$ for which $r(L) \leq 4n/9$ for any $L \in \mathcal{L}$

In 1961, Erd\H{o}s proposed a weaker version of the Strong Dirac conjecture \cite{Erd61}.  It was proved independently in 1983 both by Beck\cite{Beck83} and by Szemer\'edi and Trotter\cite{ST83}, and holds for arrangements of pseudolines.
\begin{thm}[Weak Dirac]\label{thm:weakDirac}
Let $L$ be an arrangement of $n$ pseudolines in $\mathbb{P}^2$ that do not all pass through a single point.
Then
\[r(L) = \Omega(n).\]
\end{thm}

For straight lines, the constant in Theorem \ref{thm:weakDirac} can be taken to be $1/37$, as shown in \cite{payne2012progress}.
The proof in \cite{payne2012progress} relies on Hirzebruch's inequality \cite{hirzebruch1986singularities}, an algebraic result not known to hold for pseudolines.

Traditionally, the Strong Dirac conjecture has been studied from the perspective of point sets.
In that setting, the conjecture is that any set of $n$ points includes a point incident to $n/2 - c$ lines spanned by the point set.
However, there are symmetries inherent in known extremal examples that are easier to see in the context of line arrangements.
We briefly review these examples in Section \ref{sec:strongDirac}.

In Section \ref{sec:wedges}, we describe a technique of visualizing line and pseudoline arrangements with dihedral symmetry by presenting only a single wedge, which can be used to reconstruct the entire arrangement.
This method was introduced by Eppstein on his blog \cite{Epp05}, and further developed by Berman in an investigation of simplicial pseudoline arrangements \cite{Ber08}.

In Section \ref{sec:counterexample}, we present an infinite family of arrangements of pseudolines, such that an arrangement of $n$ pseudolines from this family has no member incident to more than $4n/9 - 10/9$ vertices of the arrangement.
The family of pseudolines presented here was previously studied by Berman \cite{Ber08}, in the context of simplicial arrangements.
This is the first time an infinite family of pseudolines has been demonstrated to violate the conclusion of the Strong Dirac conjecture.

In Section \ref{sec:openProblems}, we ask a number of questions relating to the central question of what differences exist between the structure of incidences between points and lines 
versus
the structure of incidences between points and pseudolines.

\section{Strong Dirac conjecture}\label{sec:strongDirac}
In 1951, Dirac conjectured that among any set of $n$ non-collinear points, $P$, there must exist a point incident to at least $\lceil\frac{n}{2}\rceil$ lines spanned by $P$ \cite{Dir51}.
This bound can be attained for odd $n$ when the points lie on two intersecting lines.
Typically,  Dirac's original conjecture is stated in a slightly weaker form (i.e., the ``Strong Dirac'').

In \cite{AIKN11}, Akiyama et al. show that the $\lfloor\frac{n}{2}\rfloor$ bound (i.e., the Strong Dirac conjecture with $c = 0$) can be attained for all sufficiently large $n$
except those of the form $12k+11$ (which they left as an open problem).
However, there exists a family of configurations, with an arbitrarily large number of points, for which the conjecture is false for $c=0$.
This infinite family of counterexamples is due to Felsner and contains
$6k+7$ points with none incident to more than $3k+2$ spanned lines when $k$ is even, and $3k+3$ when $k$ is odd. \cite[p. 313]{BMP05}
The dual form for this family is demonstrated in Figure \ref{fig:felsner}.

\begin{figure}[htb]
 \centering
 \includegraphics[width=8cm]{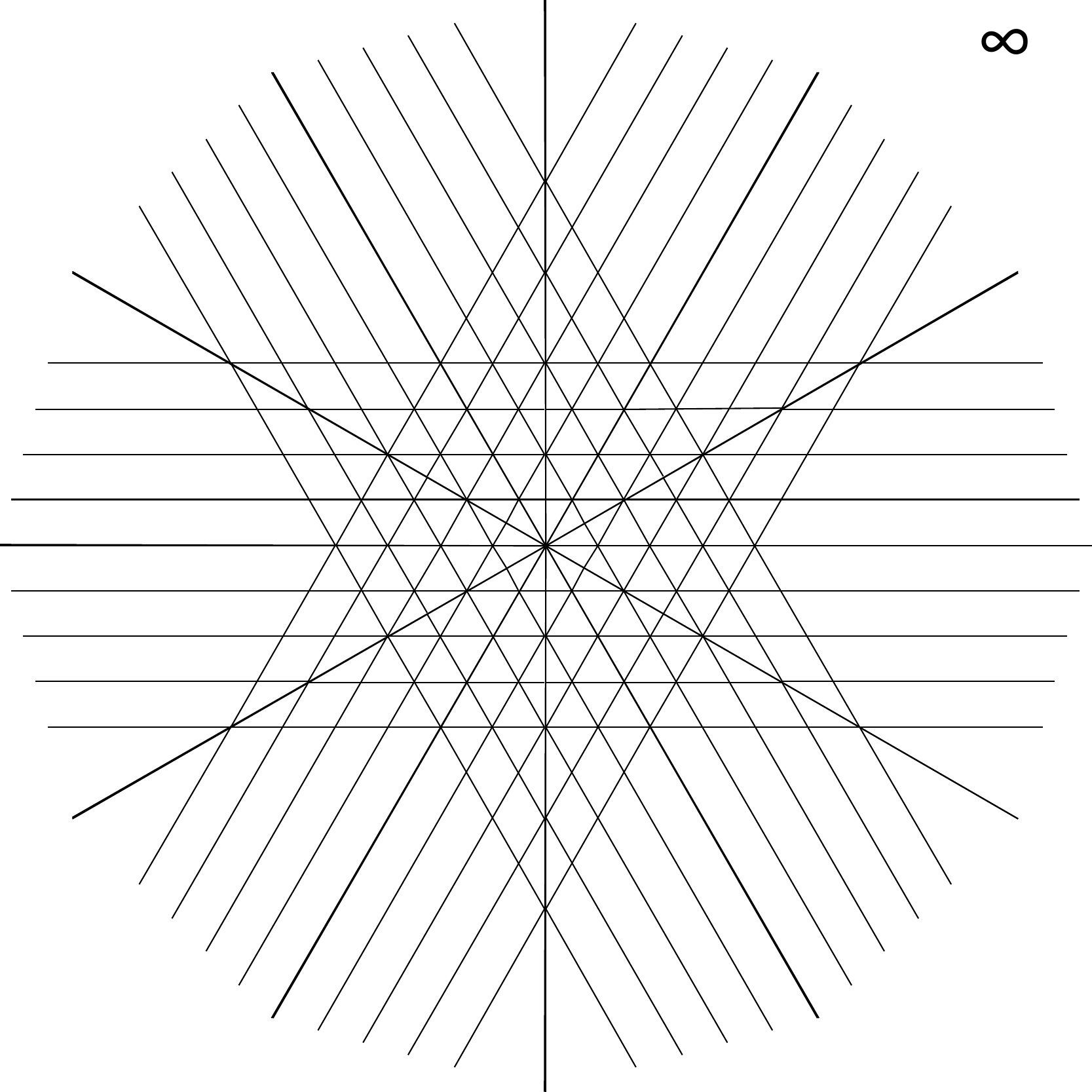}
\caption{The dual of Felsner's arrangement with $6k+7=31$ lines (including the line at infinity)
and no line incident to more than $3k+2=14$ points of intersection.}
\label{fig:felsner}
\end{figure}

No infinite family of arrangements of $n$ lines is known such that each member has fewer than $n/2 - 3/2$ intersection points, but Gr\"unbaum found several small arrangements with that property \cite{Gru72, Gru09}.
The line arrangement A[25,5] in \cite{Gru09} is the smallest member of the infinite family of pseudoline arrangements presented below.

\subsection{Wedge presentation of symmetric pseudoline arrangements}\label{sec:wedges}
A beautiful feature of Figure \ref{fig:felsner} is its symmetry.
This drawing has the symmetry of a regular hexagon (i.e., the dihedral group $D_6$).
While studying simplicial pseudoline arrangements (ones in which each planar face has three sides), Eppstein observed that arrangements with dihedral symmetry can be generated, similar to a kaleidoscope, from the contents of a single ``wedge'' \cite{Epp05}.
 Figure \ref{fig:eppstein} shows a single wedge from Felsner's arrangement.
\begin{figure}[htb]
 \centering
 \includegraphics[width=10cm]{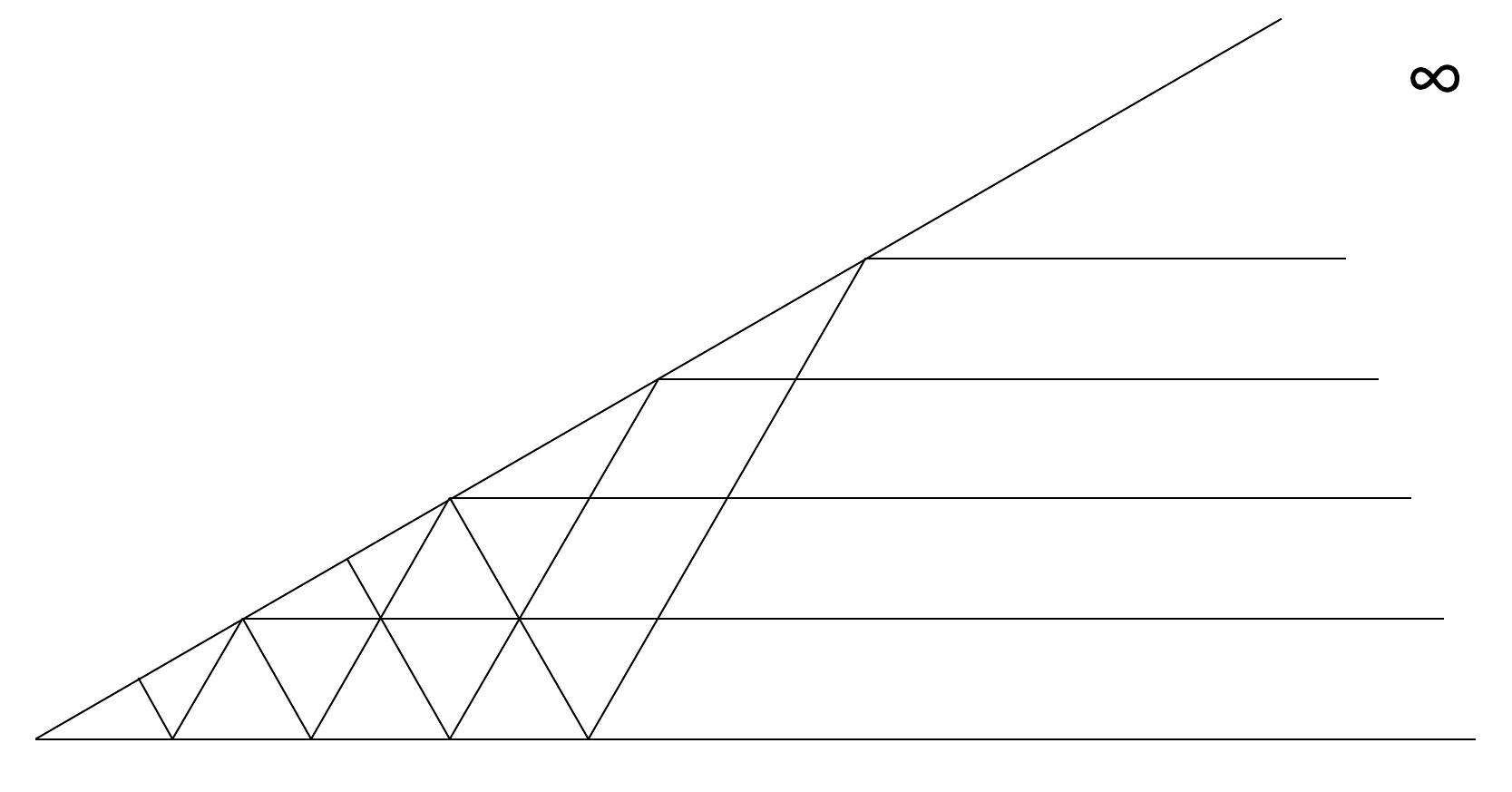}
\caption{A single wedge from Felsner's arrangement.}
\label{fig:eppstein}
\end{figure}

He noted that the entire path of a line through an arrangement
can be traced by considering that line to be ``bouncing'',
like a laser beam bouncing off mirrors, from one side of the wedge to the other.
(Notice that in Figure \ref{fig:eppstein} the beams must ``retrace'' their path after the third bounce.)
In fact for straight-line arrangements, this bouncing must follow the
\emph{law of reflection}: the angle of incidence equals the angle of reflection.
By applying basic trigonometry, one may deduce for straight-line arrangements the number and locations of the bounces
as a function of the wedge angle and the beam's initial angle of incidence.

To generate an arrangement from a wedge, the wedge must have an angle of $\pi/k$ for some positive integer $k \geq 2$.
The arrangement is produced by alternately rotating and duplicating the wedge or its mirror image,
$k$ times each, so that they fill the plane.

For pseudoline arrangements, the ``bouncing'' beams need not obey the law of reflection.

As with Felsner's arrangement a beam might retrace its path after the $\lceil\frac{k}{2}\rceil^{\rm th}$ bounce. Berman, in \cite{Ber08}, further develops Eppstein's ``kaleidoscope'' method
to construct and classify many types of symmetric simplicial pseudoline arrangements (including the one presented in Section \ref{sec:counterexample}).

\subsection{Pseudoline counterexample to Strong Dirac conjecture}\label{sec:counterexample}
\label{sec:construction}

\begin{thm}
For any $j \in \mathbb{N}^+$, there exists an arrangement of $n = 18j + 7$
pseudolines such that no pseudoline is incident to more than
$8j + 2$ vertices.
\end{thm}
We will describe the construction of a wedge for a pseudoline arrangement for arbitrary $j$, and show that it has the claimed number of pseudolines and intersection property.
We refer to Berman \cite[Fig.11]{Ber08} for a proof that the described wedge actually represents a pseudoline arrangement.

For an arbitrary $j$, the wedge angle will be $\pi/(6j+2)$.
There are four distinct symmetry classes of pseudolines, plus the line at infinity.
Two of these will be represented by the sides of the wedge; we will call these the \emph{top} and \emph{bottom} edges.
Two will be represented by beams; we will call these the \emph{red} and \emph{blue} beams. See Figure \ref{fig:pseudowedge1}.

Let $r_i^{(j)}$ be the point at which the $i^{\rm th}$ bounce of the red beam occurs along one of the two edges, counting from infinity.
Likewise, let $b_i^{(j)}$ be the point at which the $i^{\rm th}$ bounce of the blue beam occurs.
When the implied value of $j$ is obvious, we simply refer to these points as $r_i$ and $b_i$.

After the beams reach the points $r_{3j+1}$ or $b_{3j+1}$, respectively, the beams ``retrace'' their paths.
More specifically for any $k > 3j+1$, $r_k = r_{6j+2-k}$ and $b_k = b_{6j+2-k}$.

We call $r_{3j+1}$ and $b_{3j+1}$ the ``terminating points'' for their respective beams.
Prior to reaching its terminating point, every third bounce of the \emph{blue} beam coincides with a bounce of the \emph{red} beam.
For all $j$, $r_i$ = $b_{3i}$ when $i \leq j$.
The two beams are parallel to the bottom edge before the first bounce, and both $b_1$ and $r_1$ are on the top edge.
\begin{proof}
We proceed by induction.
For $j=1$, the theorem holds; the arrangement generated from this wedge contains $3(6j+2) + 1= 25$ pseudolines, each of which incident to at most $8j+2=10$ vertices. See Figure \ref{fig:pseudowedge1} for the wedge, and Figure \ref{fig:pseudo1} for the associated arrangement.

\begin{figure}[!htb]
 \centering
 \includegraphics[width=10cm]{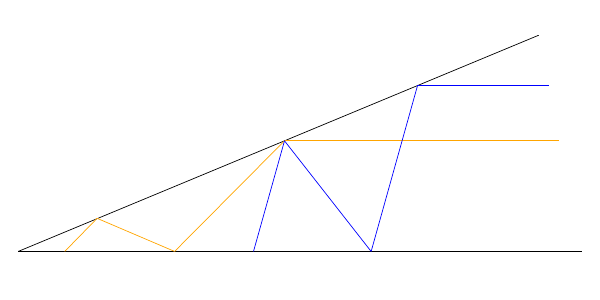}
\caption{The wedge for $j=1$, the base case for our induction.}
\label{fig:pseudowedge1}
\end{figure}

\begin{figure}[!htb]
 \centering
 \includegraphics[width=10cm]{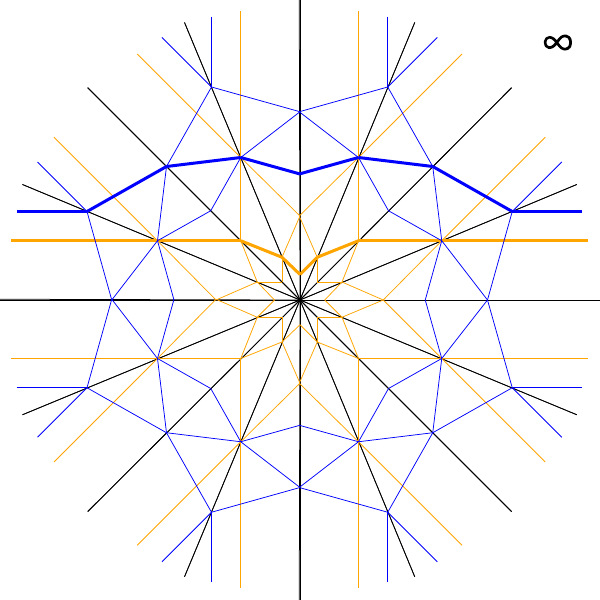}
\caption{The arrangement for $j=1$, containing $3(6j+2) + 1= 25$ pseudolines. Each pseudoline is incident to at most 10 vertices.}
\label{fig:pseudo1}
\end{figure}

Assume that the theorem holds for $j-1$.
We start with the wedge produced from the $(j-1)^{\textrm{th}}$ case,
and adjust its angle to be $\pi/(6j+2)$.
Let $r_i^{(j-1)}$ and $b_i^{(j-1)}$ be the points of the
$i^{\textrm{th}}$ bounce of the red and blue beams, respectively, from the preceding case.
We define the $r_i^{(j)}$ and $b_i^{(j)}$ as follows:
\begin{itemize}
\item For $i < 3j-1$, let $r_i^{(j)} = r_i^{(j-1)}$ and $b_i^{(j)} = b_i^{(j-1)}$.
\item For $i > 3j+3$, let $r_i^{(j)} = r_{i-6}^{(j-1)}$ and $b_i^{(j)} = b_{i-6}^{(j-1)}$.
\end{itemize}

We must specify for the $j^{\textrm{th}}$ case how to construct $\{r_{3j-1}, r_{3j}, r_{3j+1}\}$ and $\{b_{3j-1}, b_{3j}, b_{3j+1}\}$ for their respective beams.
Note that $r_{3j+2}=r_{3j}$ and $r_{3j+3}=r_{3j-1}$, and likewise for the $b_i$.

We begin with the simpler case of extending the \emph{red} beam. We place $r_{3j-1}$ on the side opposite of the wedge from $r_{3j-2}$, and continue to alternate sides when placing $r_{3j}$ and $r_{3j+1}$, each slightly closer to the corner of the wedge than the previous.

To extend the \emph{blue} beam we must cross the \emph{red} beam once, placing $b_{3j-1}$ on the opposite side of the wedge from $b_{3j-2}$.
The subsequent point, $b_{3j}$, coincides with $r_j$. 
As stated previously, $b_{3i} = r_i$ when $i\leq j$.
Lastly, place $b_{3j+1}$ at an appropriate location on the opposite side of the edge (farther from the corner than $r_{j+1}$).
With this, the construction is complete.

We must now consider the additional vertices (relative to the $(j-1)^{\textrm{th}}$ case) formed on the lines, resulting from this construction.
For the \emph{blue} lines, a total of eight vertices were added, and likewise, eight more were added for the \emph{red} lines.
The edges of the wedge correspond to the lines of the arrangement forming axes of symmetry.
For one class of axes, we added eight vertices each. To the other, we added only six each. 
(Whether the lines getting an additional six vertices correspond to the ``top'' or ``bottom'' of the wedge depends on the parity of $j$.)
See Figure \ref{fig:pseudowedge2} for the $j=2$ case, i.e., the first complete extension from the base case.

\begin{figure}[!htb]
 \centering
 \includegraphics[width=14cm]{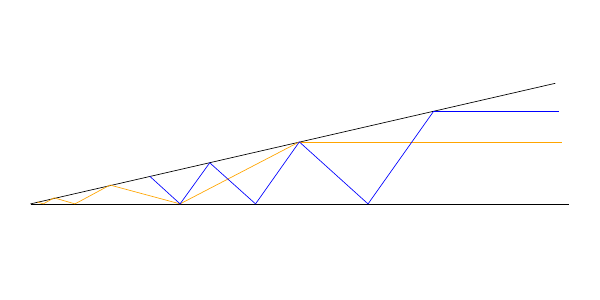}
\caption{The wedge for $j=2$.}
\label{fig:pseudowedge2}
\end{figure}

In the resulting arrangement, there will be $18j+7 = (18(j-1) + 7) + (6 \cdot 3)$ lines with none incident to more than $8j+2 = (8(j-1)+2)+8$ vertices, 
completing the inductive proof.
\end{proof}

\FloatBarrier

\section{Open problems}\label{sec:openProblems}
The main interest of the pseudoline arrangement presented here is that it shows that a natural conjecture that is widely believed to be true for straight lines is definitely false for pseudolines.
This is relevant to a more general question: how do the structural constraints on the incidences between straight lines and points differ from those on incidences between pseudolines and points?
In this section, we raise a number of specific open problems on this general theme.

\subsection{Variations on the Strong Dirac}

There is no reason to expect that $4/9$ is the best possible constant in the Weak Dirac theorem for pseudolines, and the gap between $4/9$ and the best known lower bound is quite large.
\begin{prob}
What is the supremum of values $c$ for which
\[r(L) \geq cn + o(n)\]
for all pseudoline arrangements $L$?
\end{prob}

The next question is whether (and by how much) the bound on $r(L)$ for line arrangements differs from that for pseudoline arrangements.
\begin{prob}
Is it possible to prove a lower bound on $r(L)$ that holds for line arrangements and not for pseudoline arrangements?
\end{prob}

One feature of the family of pseudoline arrangements presented in Section \ref{sec:counterexample} is that $(n-1)/3$ lines are all incident to a single vertex.
A natural question is whether this is an essential feature of any pseudoline counterexample to the Strong Dirac conjecture.
\begin{prob}
Is there an infinite family of arrangements of $n$ pseudolines, such that
\begin{itemize}
\item no vertex of any arrangement in the family is incident to $\Omega(n)$ pseudolines, and
\item no member of any arrangement is incident to more than $n/(2 + \epsilon)$ vertices for some $\epsilon > 0$?
\end{itemize}
\end{prob}

The authors are not even aware of an infinite family of pseudoline arrangements such that no vertex is incident to $\Omega(n)$ pseudolines and no pseudoline is incident to more than $cn$ vertices for some $c<1$.

Both Felsner's example, and the example presented in Section \ref{sec:counterexample} have dihedral symmetry.
Assuming that the Strong Dirac conjecture holds for line arrangements, it may be easier to prove for the special case of dihedrally symmetric line arrangements.
\begin{prob}\label{prob:dihedral}
Does the Strong Dirac conjecture hold for line arrangements with dihedral symmetry?
\end{prob}

The example presented here shows that any method used to give an affirmative answer to Problem \ref{prob:dihedral} would need to be able to distinguish between line arrangements and pseudoline arrangements.

\subsection{Dirac-Motzkin for Pseudolines}
Another classic question from incidence geometry concerns the minimum number of ordinary vertices in an arrangement of lines.
An ordinary vertex is one that is incident to exactly $2$ lines of the arrangement.
The famous conjecture on this question was known, until its recent proof by Green and Tao \cite{green2012sets}, as the Dirac-Motzkin conjecture.

\begin{thm}\label{th:GreenTao}
Let $L$ be an arrangement of $n$ lines in the plane, not all through one point.
Suppose that $n>n_0$ for a sufficiently large constant $n_0$.
Then, $L$ determines at least $n/2$ ordinary vertices.
\end{thm}

 The best result on the Dirac-Motzkin problem prior to Green and Tao's proof of Theorem \ref{th:GreenTao} was by Csima and Sawyer \cite{csima1993there}.
They showed that, if $L$ is an arrangement of $n>7$ lines in the plane, not all through one point, then $L$ determines at least $6n/13$ ordinary vertices.

A key difference between the proof of Csima and Sawyer and that of Green and Tao is that the result of Csima and Sawyer can be generalized to apply to arrangements of pseudolines in a straightforward manner \cite{lenchner2008sylvester}, but the result of Green and Tao relies on algebraic statements, including the Cayley-Bacharach theorem, that do not apply to pseudolines.

This raises the question: is the generalization of Theorem \ref{th:GreenTao} for arrangements of pseudolines true?

\begin{prob}
Is there an arrangement of $n>13$ pseudolines, not all through one point, that determines fewer than $n/2$ simple vertices?
\end{prob}

\bibliographystyle{IEEEtranS}
\bibliography{pseudolines}
\end{document}